\documentclass[11pt,a4paper]{amsart}
\usepackage{}
\usepackage[utf8]{inputenc}

\usepackage{amssymb}
\usepackage[all]{xy}
\usepackage{indentfirst}
\usepackage{amsmath,amscd}
\usepackage{tikz,tikz-cd}

\allowdisplaybreaks

\newtheorem{theorem}{Theorem}[section]
\newtheorem{proposition}[theorem]{Proposition}
\newtheorem{lemma}[theorem]{Lemma}
\newtheorem{corollary}[theorem]{Corollary}

\theoremstyle{definition}

\newtheorem{definition}[theorem]{Definition}
\newtheorem{example}[theorem]{Example}
\numberwithin{equation}{section}

\begin{document}

\author{Xueru Wu,  Yao Ma, Liangyun Chen$^*$}
\address{School of Mathematics and Statistics, Northeast Normal University,
Changchun, 130024, P.R.China}
\email{wuxr884@nenu.edu.cn.}
\address{School of Mathematics and Statistics, Northeast Normal University,
Changchun, 130024, P.R.China}
\email{may703@nenu.edu.cn.}

\address{School of Mathematics and Statistics, Northeast Normal University,
Changchun, 130024, P.R.China}
\email{chenly640@nenu.edu.cn.}

\thanks{*Corresponding author.}


\thanks{Supported by NNSF of China (No. 12071405).}

\title{Relative Rota-Baxter operators of nonzero weights on Lie Triple Systems}

\begin{abstract}
In this paper, we introduce the notion of a relative Rota-Baxter operator of weight $\lambda$ on a Lie triple system with respect to an action on another Lie triple system, which can be characterized by the graph of their semidirect product. We also establish a cohomology theory for a relative Rota-Baxter operator of weight $\lambda$ on Lie triple systems and use the first cohomology group to classify infinitesimal deformations.
\end{abstract}
\maketitle

\section{Introduction}

Lie triple systems were introduced by Cartan in his studies on Riemannian geometry. Since then, Jacobson studied Lie triple systems by an algebraic method in \cite{Jac1,Jac}. The present formulation is due to Yamaguti \cite{Yam}. Lie triple systems have important applications in physics, such as in quantum mechanics and numerical analysis of differential equations. For more results see \cite{CHMM,Ku,Zhang1}. People often refer to Lie algebras when considering the related problems of Lie triple systems, since Lie triple systems are closely related to Lie algebras.

Rota-Baxter associative algebras originated from the probability study in \cite{Baxter}. In the Lie algebra context, Kupershmidt introduced the notion of an $\mathcal{O}$-operator (also called relative Rota-Baxter operator) in \cite{Kupershmidt} to better understand the classical Yang-Baxter equation. (Relative) Rota-Baxter operators on Lie algebras and associative algebras have important applications in various fields, such as the classical Yang-Baxter equation and integrable systems \cite{Bai,Kupershmidt}, splitting of operads \cite{BBGN,PBG}, double Lie algebras \cite{GK}, and etc. In \cite{Kupershmidt}, the author showed the concept of $\mathcal{O}$-operators (relative Rota-Baxter operator) on Lie algebras.  Then  $\mathcal{O}$-operators have been used to define cohomology on 3-Lie algebras, Lie triple system and Lie-Yamaguti algebras, see \cite{CHMM,SZ,THS}.

In \cite{Gerstenhaber}, the author introduced the deformation of algebraic structures in his study associative algebras. Then, the deformation theory was then extended to Lie algebras \cite{NR}, Hom-type algebras \cite{MS}, $3$-Lie algebras \cite{Figueroa} and Lie triple systems \cite{YBB}. Recently, deformations of morphisms and Rota-Baxter operators were deeply studied \cite{ABM,Das,FZ,TBGS,WZ}. Also, (Relative) Rota-Baxter operators of nonzero weights on 3-Lie algebras and matrix
algebras of order three have been introduced in \cite{BGLW,GG,HSZ}. The main purpose of this paper is to introduce the notion of a relative Rota-Baxter operators of nonzero weight $\lambda$ from a Lie triple system $\mathfrak{L}'$ to a Lie triple system $\mathfrak{L}$ with respect to an action $\theta$, and characterize it using the graph of the semidirect product Lie triple systems.

The paper is organized as follows. In Section 2, we introduce the notion of a relative Rota-Baxter operator of weight $\lambda$ from a Lie triple system $\mathfrak{L}'$ to a Lie triple system $\mathfrak{L}$ with respect to an action $\theta$. In Section 3, we establish a cohomology theory for relative Rota-Baxter operator of weight $\lambda$ on Lie triple systems and classify infinitesimal deformations using the first cohomology group.

In this paper, all Lie triple systems are defined over an arbitrary field $\mathbb{F}$ of characteristic 0.

\section{Relative Rota-Baxter operators of weight $\lambda$ on Lie triple systems}

In this section, we first introduce the notion of actions of Lie triple systems, which give rise to the semidirect product Lie triple systems. Then we introduce the notion of relative Rota-Baxter operators of weight $\lambda$ on Lie triple system, which can be characterized by the graphs of the semidirect product Lie triple systems. Finally, we establish the relation between relative Rota-Baxter operators of weight $\lambda$ on Lie triple systems and Nijenhuis operators on Lie triple systems.
\begin{definition}{\rm\cite{Jac1}}
A Lie triple system is a vector space $\mathfrak{L}$ endowed with a trilinear operation $ [ \cdot, \cdot, \cdot ]_\mathfrak{L}: \mathfrak{L}\times \mathfrak{L}\times \mathfrak{L}\longrightarrow \mathfrak{L}$ satisfying
\begin{align}
&[a,a,b]_\mathfrak{L}=0,\label{392.1}\\
&[a,b,c]_\mathfrak{L}+[b,c,a]_\mathfrak{L}+[c,a,b]_\mathfrak{L}=0,\label{392.2}\\
&[a,b,[c,d,e]_\mathfrak{L}]_\mathfrak{L}=[[a,b,c]_\mathfrak{L},d,e]_\mathfrak{L}
+[c,[a,b,d]_\mathfrak{L},e]_\mathfrak{L}+[c,d,[a,b,e]_\mathfrak{L}]_\mathfrak{L},\label{392.3}
\end{align}
for all $a, b, c, d, e \in \mathfrak{L}$.
\end{definition}

\begin{definition}{\rm\cite{Yam}}\label{def392.2}
An $\mathfrak{L}$-module is a vector space $V$ with a bilinear map
\begin{align*}
\theta:\mathfrak{L} \times \mathfrak{L}& \longrightarrow {\rm End}(V)\\
                                 (a,b)& \longrightarrow \theta(a,b),
\end{align*}
such that the following conditions hold:
\begin{align*}
&\theta(c,d)\theta(a,b)-\theta(b,d)\theta(a,c)-\theta(a,[b,c,d]_\mathfrak{L})+D(b,c)\theta(a,d)=0,\\
&\theta(c,d)D(a,b)-D(a,b)\theta(c,d)+\theta([a,b,c]_\mathfrak{L},d)+\theta(c,[a,b,d]_\mathfrak{L})=0,
\end{align*}
where
\begin{align*}
D(a,b)=\theta(b,a)-\theta(a,b),
\end{align*}
for all $a,b,c,d\in \mathfrak{L}.$ Also $\theta$ is called a representation of $\mathfrak{L}$ on $V.$
\end{definition}

In particular, if
\begin{align*}
\theta:\mathfrak{L} \times \mathfrak{L}& \longrightarrow {\rm End}(\mathfrak{L})\\
                                 (a,b)& \longrightarrow \theta(a,b),
\end{align*}
with $\theta(a,b)c=[c,a,b]_\mathfrak{L},$ and $D(a,b)c=\theta(b,a)c-\theta(a,b)c=[a,b,c]_\mathfrak{L}.$ Then $\theta$ is called an adjoint representation of $\mathfrak{L}$ on itself.

Now, let us give the definition of a relative Rota-Baxter operator of weight $\lambda\in \mathbb{F}$ from a Lie triple system to another Lie triple system. We recall a derived algebra of a Lie triple system $(\mathfrak{L},[\cdot,\cdot,\cdot]_\mathfrak{L})$ is a subsystem $[\mathfrak{L},\mathfrak{L},\mathfrak{L}]_\mathfrak{L}$, and denoted by $\mathfrak{L}^1$. The subspace
\begin{align*}
C(\mathfrak{L})=\{x\in \mathfrak{L}~|~[x,y,z]_\mathfrak{L}=0,~\forall~y,z\in \mathfrak{L}\}
\end{align*}
is the center of $(\mathfrak{L},[\cdot,\cdot,\cdot]_\mathfrak{L})$.

\begin{definition}\label{def2112.4}
Let $(\mathfrak{L},[\cdot,\cdot,\cdot]_{\mathfrak{L}})$ and $(\mathfrak{L}^\prime,[\cdot,\cdot,\cdot]_{\mathfrak{L}^\prime})$ be two Lie triple systems. Let $\theta:\mathfrak{L}\times\mathfrak{L}\rightarrow {\rm End}(\mathfrak{L}^\prime)$ be a representation of the Lie triple system $\mathfrak{L}$ on the vector space $\mathfrak{L}^\prime$. If for all $x,y\in \mathfrak{L}$, $u,v,w\in \mathfrak{L}^\prime$,
\begin{align*}
\theta(x,y)u\in C(\mathfrak{L}^\prime),\\
\theta(x,y)[u,v,w]_{\mathfrak{L}'}=0,
\end{align*}
then $\theta$ is called an action of $(\mathfrak{L},[\cdot,\cdot,\cdot]_{\mathfrak{L}})$ on $(\mathfrak{L}^\prime,[\cdot,\cdot,\cdot]_{\mathfrak{L}^\prime})$. We denote an action by $(\mathfrak{L}^\prime,\theta)$.
\end{definition}

\begin{example}
Let $(\mathfrak{L},[\cdot,\cdot,\cdot]_{\mathfrak{L}})$ be a Lie triple system. If $\mathfrak{L}$ satisfies $\mathfrak{L}^1\in C(\mathfrak{L})$, then the adjoint representation $\theta:\mathfrak{L}\times\mathfrak{L}\rightarrow {\rm End}(\mathfrak{L})$ is an action of $\mathfrak{L}$ on itself.
\end{example}

\begin{definition}
Let $\theta:\mathfrak{L}\times\mathfrak{L}\rightarrow {\rm End}(\mathfrak{L}')$ be an action of a Lie triple system $(\mathfrak{L},[\cdot,\cdot,\cdot]_{\mathfrak{L}})$ on a Lie triple system $(\mathfrak{L}^\prime,[\cdot,\cdot,\cdot]_{\mathfrak{L}^\prime})$. A linear map $T:\mathfrak{L}^\prime\rightarrow \mathfrak{L}$ is called a relative Rota-Baxter operator of weight $\lambda\in \mathbb{F}$ from a Lie triple system $(\mathfrak{L}^\prime,[\cdot,\cdot,\cdot]_{\mathfrak{L}^\prime})$ to a Lie triple system $(\mathfrak{L},[\cdot,\cdot,\cdot]_{\mathfrak{L}})$ with respect to an action $\theta$ if
\begin{align}\label{2112.7}
[Tu,Tv,Tw]_{\mathfrak{L}}=T\Big(D(Tu,Tv)w-\theta(Tu,Tw)v+\theta(Tv,Tw)u+\lambda[u,v,w]_{\mathfrak{L}^\prime}\Big),
\end{align}
for all $u,v,w\in \mathfrak{L}^\prime$.
\end{definition}

\begin{definition}
Let $T$ and $T^\prime$ be two relative Rota-Baxter operators of weight $\lambda$ from a Lie triple system $(\mathfrak{L}^\prime,[\cdot,\cdot,\cdot]_{\mathfrak{L}^\prime})$ to a Lie triple system $(\mathfrak{L},[\cdot,\cdot,\cdot]_{\mathfrak{L}})$ with respect to an action $\theta$. A homomorphism from $T$ to $T^\prime$ consists of homomorphisms $\psi_{\mathfrak{L}}:\mathfrak{L} \rightarrow \mathfrak{L}$ and $\psi_{\mathfrak{L}^\prime}:\mathfrak{L}^\prime \rightarrow \mathfrak{L}^\prime$ such that
\begin{align}
\psi_{\mathfrak{L}}\circ T=& T^\prime\circ \psi_{\mathfrak{L}^\prime},\label{2112.8}\\
\psi_{\mathfrak{L}^\prime}(\theta(x,y)u)=&\theta(\psi_{\mathfrak{L}}(x),\psi_{\mathfrak{L}}(y))\psi_{\mathfrak{L}^\prime}(u)
,~~\forall x,y\in \mathfrak{L},~u\in \mathfrak{L}^\prime\label{2112.9}\\
\psi_{\mathfrak{L}^\prime}(D(x,y)u)=&D(\psi_{\mathfrak{L}}(x),\psi_{\mathfrak{L}}(y))\psi_{\mathfrak{L}^\prime}(u),\label{2112.10}
~~\forall x,y\in \mathfrak{L},~u\in \mathfrak{L}^\prime.
\end{align}
In particular, if both $\psi_{\mathfrak{L}}$ and $\psi_{\mathfrak{L}^\prime}$ are invertible, $(\psi_{\mathfrak{L}},\psi_{\mathfrak{L}^\prime})$ is called an isomorphism from $T$ to $T^\prime$.
\end{definition}

Let $\theta:\mathfrak{L}\times\mathfrak{L}\rightarrow {\rm End}(\mathfrak{L}^\prime)$ be an action of a Lie triple system $(\mathfrak{L},[\cdot,\cdot,\cdot]_{\mathfrak{L}})$ on a Lie triple system $(\mathfrak{L}^\prime,[\cdot,\cdot,\cdot]_{\mathfrak{L}^\prime})$. Define $[\cdot,\cdot,\cdot]_{\theta}$ on $\mathfrak{L}\oplus \mathfrak{L}^\prime$ by
\begin{align*}
[x+u,y+v,z+w]_{\theta}=[x,y,z]_{\mathfrak{L}}+D(x,y)w+\theta(y,z)u-\theta(x,z)v+\lambda[u,v,w]_{\mathfrak{L}^\prime},
\end{align*}
for all $x,y,z\in \mathfrak{L}$, $u,v,w\in \mathfrak{L}^\prime$.

\begin{proposition}
Keep notations as  above, then $(\mathfrak{L}\oplus \mathfrak{L}^\prime,[\cdot,\cdot,\cdot]_{\theta})$ is a Lie triple system, which is called the semidirect product of the Lie triple system $(\mathfrak{L},[\cdot,\cdot,\cdot]_{\mathfrak{L}})$ on the Lie triple system $(\mathfrak{L}^\prime,[\cdot,\cdot,\cdot]_{\mathfrak{L}^\prime})$ with respect to the action $\theta$, and denoted by $\mathfrak{L}\ltimes_{\theta}\mathfrak{L}^\prime.$
\end{proposition}
\begin{proof}
To prove $(\mathfrak{L}\oplus \mathfrak{L}^\prime,[\cdot,\cdot,\cdot]_{\theta})$ is a Lie triple system, we only need to show Eqs. (\ref{392.1})-(\ref{392.3}) hold. First, we show that $(\mathfrak{L}\oplus \mathfrak{L}^\prime,[\cdot,\cdot,\cdot]_{\theta})$ satisfies Eq. (\ref{392.1}), for all $x,y,z\in \mathfrak{L}$, $u,v,w\in \mathfrak{L}'$,
\begin{align*}
[x+u,x+u,z+w]_\theta=[x,x,z]_\mathfrak{L}+D(x,x)w+\theta(x,z)u-\theta(x,z)u+\lambda[u,u,w]_{\mathfrak{L}^\prime}=0.
\end{align*}

Then, for Eq. (\ref{392.2}), we have
\begin{align*}
&[x+u,y+v,z+w]_\theta+[y+v,z+w, x+u]_\theta+[z+w, x+u,y+v]_\theta\\
=&[x,y,z]_\mathfrak{L}+D(x,y)w+\theta(y,z)u-\theta(x,z)v+\lambda[u,v,w]_{\mathfrak{L}^\prime}\\
&+[y,z,x]_\mathfrak{L}+D(y,z)u+\theta(z,x)v-\theta(y,x)w+\lambda[v,w,u]_{\mathfrak{L}^\prime}\\
&+[z,x,y]_\mathfrak{L}+D(z,x)v+\theta(x,y)w-\theta(z,y)u+\lambda[w,u,v]_{\mathfrak{L}^\prime}\\
=&0.
\end{align*}

At last, we show that Eq. (\ref{392.3}) holds
\begin{align*}
&[a+p,b+q,[x+u,y+v,z+w]_\theta]_\theta-[[a+p,b+q,x+u]_\theta,y+v,z+w]_\theta\\
&-[x+u,[a+p,b+q,y+v]_\theta,z+w]_\theta-[x+u,y+v,[a+p,b+q,z+w]_\theta]_\theta\\
=&[a+p,b+q,[x,y,z]_\mathfrak{L}+D(x,y)w+\theta(y,z)u-\theta(x,z)v+\lambda[u,v,w]_{\mathfrak{L}^\prime}]_\theta\\
&-[[a,b,x]_\mathfrak{L}+D(a,b)u+\theta(b,x)p-\theta(a,x)q+\lambda[p,q,u]_{\mathfrak{L}^\prime},y+v,z+w]_\theta\\
&-[x+u,[a,b,y]_\mathfrak{L}+D(a,b)v+\theta(b,y)p-\theta(a,y)q+\lambda[p,q,v]_{\mathfrak{L}^\prime},z+w]_\theta\\
&-[x+u,y+v,[a,b,z]_\mathfrak{L}+D(a,b)w+\theta(b,z)p-\theta(a,z)q+\lambda[p,q,w]_{\mathfrak{L}^\prime}]_\theta\\
=&[a,b,[x,y,z]_\mathfrak{L}]_\mathfrak{L}+D(a,b)D(x,y)w+D(a,b)\theta(y,z)u-D(a,b)\theta(x,z)v\\
&+\lambda D(a,b)[u,v,w]_{\mathfrak{L}^\prime}+\theta(b,[x,y,z]_\mathfrak{L})p-\theta(a,[x,y,z]_\mathfrak{L})q
+\lambda[p,q,D(x,y)w]_{\mathfrak{L}^\prime}\\
&+\lambda[p,q,\theta(y,z)u]_{\mathfrak{L}^\prime}-\lambda[p,q,\theta(x,z)v]_{\mathfrak{L}^\prime}
+\lambda[p,q,\lambda[u,v,w]_{\mathfrak{L}^\prime}]_{\mathfrak{L}^\prime}\\
&-[[a,b,x]_\mathfrak{L},y,z]_\mathfrak{L}-D([a,b,x]_\mathfrak{L},y)w-\theta(y,z)D(a,b)u-\theta(y,z)\theta(b,x)p\\
&+\theta(y,z)\theta(a,x)q-\lambda\theta(y,z)[p,q,u]_{\mathfrak{L}^\prime}+\theta([a,b,x]_\mathfrak{L},z)v
-\lambda[D(a,b)u, v,w]_{\mathfrak{L}^\prime}\\
&-\lambda[\theta(b,x)p, v,w]_{\mathfrak{L}^\prime}+\lambda[\theta(a,x)q,v,w]_{\mathfrak{L}^\prime}
-\lambda[\lambda[p,q,u]_{\mathfrak{L}^\prime},v,w]_{\mathfrak{L}^\prime}\\
&-[x,[a,b,y]_\mathfrak{L},z]_\mathfrak{L}-D(x,[a,b,y]_\mathfrak{L})w-\theta([a,b,y]_\mathfrak{L},z)u+\theta(x,z)D(a,b)v\\
&+\theta(x,z)\theta(b,y)p-\theta(x,z)\theta(a,y)q+\lambda\theta(x,z)[p,q,v]_{\mathfrak{L}^\prime}
-\lambda[u,\theta(b,y)p,w]_{\mathfrak{L}^\prime}\\
&-\lambda[u,D(a,b)v,w]_{\mathfrak{L}^\prime}+\lambda[u,\theta(a,y)q,w]_{\mathfrak{L}^\prime}
+\lambda[u,\lambda[p,q,v]_{\mathfrak{L}^\prime},w]_{\mathfrak{L}^\prime}\\
&-[x,y,[a,b,z]_\mathfrak{L}]_\mathfrak{L}-D(x,y)\theta(b,z)p-D(x,y)D(a,b)w+D(x,y)\theta(a,z)q\\
&-\theta(y,[a,b,z]_\mathfrak{L})u+\theta(x,[a,b,z]_\mathfrak{L})v-\lambda D(x,y)[p,q,w]_{\mathfrak{L}^\prime}
-\lambda[u,v,\theta(b,z)p]_{\mathfrak{L}^\prime}\\
&-\lambda[u,v,D(a,b)w]_{\mathfrak{L}^\prime}+\lambda[u,v,\theta(a,z)q]_{\mathfrak{L}^\prime}
-\lambda[u,v,\lambda[p,q,w]_{\mathfrak{L}^\prime}]_{\mathfrak{L}^\prime}\\
=&0,
\end{align*}
which implies that $(\mathfrak{L}\oplus \mathfrak{L}^\prime,[\cdot,\cdot,\cdot]_{\theta})$ is a Lie triple system.
\end{proof}

\begin{theorem}
Let $\theta:\mathfrak{L}\times\mathfrak{L}\rightarrow {\rm End}(\mathfrak{L}^\prime)$ be an action of a Lie triple system $(\mathfrak{L},[\cdot,\cdot,\cdot]_{\mathfrak{L}})$ on a Lie triple system $(\mathfrak{L}^\prime,[\cdot,\cdot,\cdot]_{\mathfrak{L}^\prime})$. Then a linear map $T:\mathfrak{L}^\prime \rightarrow \mathfrak{L}$ is a relative Rota-Baxter operator of weight $\lambda$ if and only if the graph
\begin{align*}
Gr(T)=\{Tu+u~|~u\in \mathfrak{L}^\prime\}
\end{align*}
is a subsystem of the Lie triple system $\mathfrak{L}\ltimes_{\theta}\mathfrak{L}^\prime$.
\end{theorem}
\begin{proof}
Let $T:\mathfrak{L}^\prime \rightarrow \mathfrak{L}$ be a linear map. For all $u,v,w\in \mathfrak{L}^\prime$, we have
\begin{align*}
&[Tu+u,Tv+v,Tw+w]_\theta\\
=&[Tu,Tv,Tw]_\mathfrak{L}+D(Tu,Tv)w+\theta(Tv,Tw)u-\theta(Tu,Tw)v+\lambda[u,v,w]_{\mathfrak{L}^\prime},
\end{align*}
which implies that the graph $Gr(T)=\{Tu+u~|~u\in \mathfrak{L}^\prime\}$ is a subsystem of the Lie triple system $\mathfrak{L}\ltimes_{\theta}\mathfrak{L}^\prime$ if and only if $T$ satisfies
\begin{align*}
[Tu,Tv,Tw]_\mathfrak{L}=T\Big(D(Tu,Tv)w+\theta(Tv,Tw)u-\theta(Tu,Tw)v+\lambda[u,v,w]_{\mathfrak{L}^\prime}\Big),
\end{align*}
which means that $T$ is a relative Rota-Baxter operator of weight $\lambda$.
\end{proof}

Since the graph $Gr(T)$ is isomorphic to $\mathfrak{L}^\prime$ as vector spaces, then there is an induced Lie triple system structure on $\mathfrak{L}^\prime$.

\begin{corollary}
Let $T:\mathfrak{L}^\prime \rightarrow \mathfrak{L}$ be a relative Rota-Baxter operator of weight $\lambda$ from a Lie triple system $(\mathfrak{L}^\prime,[\cdot,\cdot,\cdot]_{\mathfrak{L}^\prime})$ to a Lie triple system $(\mathfrak{L},[\cdot,\cdot,\cdot]_{\mathfrak{L}})$ with respect to an action $\theta$. Then $(\mathfrak{L}^\prime,[\cdot,\cdot,\cdot]_T)$ is a Lie triple system, called the descendent Lie triple system of $T$, where
\begin{align*}
[u,v,w]_T=D(Tu,Tv)w+\theta(Tv,Tw)u-\theta(Tu,Tw)v+\lambda[u,v,w]_{\mathfrak{L}^\prime},~~\forall~u,v,w\in \mathfrak{L}^\prime.
\end{align*}
Moreover, $T$ is a homomorphism from $(\mathfrak{L}^\prime,[\cdot,\cdot,\cdot]_T)$ to $(\mathfrak{L},[\cdot,\cdot,\cdot]_{\mathfrak{L}})$.
\end{corollary}

In the sequel, we give the relationship between relative Rota-Baxter operators of weight $\lambda$ and Nijenhuis operators. Recall from {\rm\cite{CHMM}} that a Nijenhuis operator on a Lie triple system $(\mathfrak{L},[\cdot,\cdot,\cdot]_{\mathfrak{L}})$ is a linear map $N:\mathfrak{L} \rightarrow \mathfrak{L}$ satisfying
\begin{align*}
[Nx,Ny,Nz]_\mathfrak{L}=&N[Nx,Ny,z]_\mathfrak{L}+N[x,Ny,Nz]_\mathfrak{L}+N[Nx,y,Nz]_\mathfrak{L}-N^2[Nx,y,z]_\mathfrak{L}\notag\\
&-N^2[x,Ny,z]_\mathfrak{L}-N^2[x,y,Nz]_\mathfrak{L}+N^3[x,y,z]_\mathfrak{L},
\end{align*}
for all $x,y,z\in \mathfrak{L}$.

\begin{proposition}
Let $\theta:\mathfrak{L}\times\mathfrak{L}\rightarrow {\rm End}(\mathfrak{L}^\prime)$ be an action of a Lie triple system $(\mathfrak{L},[\cdot,\cdot,\cdot]_{\mathfrak{L}})$ on a Lie triple system $(\mathfrak{L}^\prime,[\cdot,\cdot,\cdot]_{\mathfrak{L}^\prime})$. Then a linear map $T:\mathfrak{L}^\prime \rightarrow \mathfrak{L}$ is a relative Rota-Baxter operator of weight $\lambda$ if and only if
\begin{align*}
\widetilde{T}=
\left(
\begin{matrix}
{\rm id}&T\\
0&0
\end{matrix}
\right):\mathfrak{L}\oplus \mathfrak{L}^\prime\rightarrow \mathfrak{L}\oplus \mathfrak{L}^\prime
\end{align*}
is a Nijenhuis operator acting on the Lie triple system $\mathfrak{L}\ltimes_{\theta}\mathfrak{L}^\prime$.
\end{proposition}
\begin{proof}
For all $x,y,z\in \mathfrak{L}$, $u,v,w\in \mathfrak{L}^\prime$, on the one hand, we have
\begin{align*}
&[\widetilde{T}(x+u),\widetilde{T}(y+v),\widetilde{T}(z+w)]_\theta\\
=&[x+Tu,y+Tv,z+Tw]_\theta\\
=&[x,y,z]_\mathfrak{L}+[Tu,y,z]_\mathfrak{L}+[x,Tv,z]_\mathfrak{L}+[x,y,Tw]_\mathfrak{L}+[x,Tv,Tw]_\mathfrak{L}
+[Tu,y,Tw]_\mathfrak{L}\\
&+[Tu,Tv,z]_\mathfrak{L}+[Tu,Tv,Tw]_\mathfrak{L}.
\end{align*}

On the other hand, since $\widetilde{T}^2=\widetilde{T}$, we have
\begin{align*}
&\widetilde{T}\Big([\widetilde{T}(x+u),\widetilde{T}(y+v),z+w]_\theta+[\widetilde{T}(x+u),y+v,\widetilde{T}(z+w)]_\theta
+[x+u,\widetilde{T}(y+v),\widetilde{T}(z+w)]_\theta\Big)\\
&-\widetilde{T}^2\Big([\widetilde{T}(x+u),y+v,z+w]_\theta+[x+u,y+v,\widetilde{T}(z+w)]_\theta
+[x+u,\widetilde{T}(y+v),z+w]_\theta\Big)\\
&+\widetilde{T}^3\Big([x+u,y+v,z+w]_\theta\Big)\\
=&[x,y,z]_\mathfrak{L}+[Tu,y,z]_\mathfrak{L}+[x,Tv,z]_\mathfrak{L}+[x,y,Tw]_\mathfrak{L}+[x,Tv,Tw]_\mathfrak{L}
+[Tu,y,Tw]_\mathfrak{L}\\
&+[Tu,Tv,z]_\mathfrak{L}+T\Big(D(Tu,Tv)w+\theta(Tv,Tw)u-\theta(Tu,Tw)v+\lambda[u,v,w]_{\mathfrak{L}^\prime}\Big),
\end{align*}
which implies that $\widetilde{T}$ is a Nijenhuis operator on the Lie triple system $\mathfrak{L}\ltimes_{\theta}\mathfrak{L}^\prime$ if and only if Eq. (\ref{2112.7}) holds.
\end{proof}

We will end this section with some examples of relative Rota-Baxter operators on Lie triple systems.

\begin{proposition}\label{prop2112.13}
Let $(\mathfrak{L},[\cdot,\cdot,\cdot]_{\mathfrak{L}})$ be a Lie triple system such that the adjoint representation $\theta: \mathfrak{L}\times\mathfrak{L} \rightarrow {\rm End}(\mathfrak{L})$ is an action of the Lie triple system $(\mathfrak{L},[\cdot,\cdot,\cdot]_{\mathfrak{L}})$ on itself. Let $\mathfrak{L}^\prime$ be an abelian subsystem of Lie triple system $\mathfrak{L}$ and satisfy $\mathfrak{L}^1\cap \mathfrak{L}^\prime=0$. Let $\mathfrak{h}$ be a compliment of $\mathfrak{L}^\prime$ such that $\mathfrak{L}=\mathfrak{h}\oplus\mathfrak{L}^\prime$ as vector spaces. Then the projection $P:\mathfrak{L}\rightarrow \mathfrak{L}$ onto the subspace $\mathfrak{L}^\prime$, i.e., $P(k+u)=u$, for all $k\in \mathfrak{h}$, $u\in \mathfrak{L}^\prime$, is a relative Rota-Baxter operator of weight $\lambda$ from $\mathfrak{L}$ to $\mathfrak{L}$ with respect to the adjoint action $\theta$.
\end{proposition}
\begin{proof}
For all $x,y,z\in \mathfrak{L}$, denote by $x^\prime,y^\prime,z^\prime$ their images under the projection $P$. Since $\mathfrak{L}^\prime$ is abelian and $\mathfrak{L}^1\cap \mathfrak{L}^\prime=0$, we have
\begin{align*}
&[P(x),P(y),P(z)]_\mathfrak{L}-P\Big([P(x),P(y),z]_\mathfrak{L}+[P(x),y,P(z)]_\mathfrak{L}+[x,P(y),P(z)]_\mathfrak{L}
+\lambda[x,y,z]_{\mathfrak{L}}\Big)\\
=&[x^\prime,y^\prime,z^\prime]_\mathfrak{L}-P\Big([x^\prime,y^\prime,z]_\mathfrak{L}+[x^\prime,y,z^\prime]_\mathfrak{L}
+[x,y^\prime,z^\prime]_\mathfrak{L}+\lambda[x,y,z]_\mathfrak{L}\Big)\\
=&0,
\end{align*}
which implies that the projection $P$ is a relative Rota-Baxter operator of weight $\lambda$ from $(\mathfrak{L},[\cdot,\cdot,\cdot]_{\mathfrak{L}})$ to $(\mathfrak{L},[\cdot,\cdot,\cdot]_{\mathfrak{L}})$ with respect to the adjoint action $\theta$.
\end{proof}

\begin{example}
Let $(\mathfrak{L},[\cdot,\cdot,\cdot]_\mathfrak{L})$ be a 3-dimensional Lie triple system with a basis $\{e_1.e_2,e_3\}$ and the nonzero multiplication is defined by
\begin{align*}
[e_1,e_2,e_1]_\mathfrak{L}=e_3.
\end{align*}
The center of $(\mathfrak{L},[\cdot,\cdot,\cdot]_\mathfrak{L})$ is the subspace generated by $\{e_3\}$. It is obvious that the adjoint representation $\theta:\mathfrak{L}\times\mathfrak{L}\rightarrow {\rm End}(\mathfrak{L})$ is an action of $(\mathfrak{L},[\cdot,\cdot,\cdot]_\mathfrak{L})$ on itself. Let $\mathfrak{L}'$ be an abelian subsystem of $(\mathfrak{L},[\cdot,\cdot,\cdot]_\mathfrak{L})$ generated by $\{e_1\}$. By Proposition \ref{prop2112.13}, the projection $P:\mathfrak{L}\rightarrow \mathfrak{L}$ given by
\begin{equation*}
\left\{
\begin{array}{lr}P(e_1)=e_1,&\\
P(e_2)=0,&\\
P(e_3)=0,&
\end{array}
\right.
\end{equation*}
is a relative Rota-Baxter operator of weight $\lambda$ from $(\mathfrak{L},[\cdot,\cdot,\cdot]_\mathfrak{L})$ to $(\mathfrak{L},[\cdot,\cdot,\cdot]_\mathfrak{L})$ with respect to the adjoint action $\theta$.
\end{example}

\begin{example}
Let $(\mathfrak{L},[\cdot,\cdot,\cdot]_\mathfrak{L})$ be a 4-dimensional Lie triple system with a basis $\{e_1.e_2,e_3,e_4\}$ and the nonzero multiplication is defined by
\begin{align*}
[e_1,e_2,e_1]_\mathfrak{L}=e_4.
\end{align*}
The center of $(\mathfrak{L},[\cdot,\cdot,\cdot]_\mathfrak{L})$ is the subspace generated by $\{e_3,e_4\}$. It is obvious that the adjoint representation $\theta:\mathfrak{L}\times\mathfrak{L}\rightarrow {\rm End}(\mathfrak{L})$ is an action of $(\mathfrak{L},[\cdot,\cdot,\cdot]_\mathfrak{L})$ on itself. Let $\mathfrak{L}'$ be an abelian subsystem of $(\mathfrak{L},[\cdot,\cdot,\cdot]_\mathfrak{L})$ generated by $\{e_2,e_3\}$. By Proposition \ref{prop2112.13}, the projection $P:\mathfrak{L}\rightarrow \mathfrak{L}$ given by
\begin{equation*}
\left\{
\begin{array}{lr}P(e_1)=0,&\\
P(e_2)=e_2,&\\
P(e_3)=e_3,&\\
P(e_4)=0,&
\end{array}
\right.
\end{equation*}
is a relative Rota-Baxter operator of weight $\lambda$ from $(\mathfrak{L},[\cdot,\cdot,\cdot]_\mathfrak{L})$ to $(\mathfrak{L},[\cdot,\cdot,\cdot]_\mathfrak{L})$ with respect to the adjoint action $\theta$.
\end{example}

\section{Cohomologies of relative Rota-Baxter operators and infinitesimal deformations}

In this section, we define a cohomology of a relative Rota-Baxter operator $T$ of weight $\lambda$.  This cohomology will be
used to study infinitesimal deformations of $T$.

Let $(V,\theta)$ be a representation of a Lie triple system $(\mathfrak{L},[\cdot,\cdot,\cdot]_{\mathfrak{L}})$. Denote by
\begin{align*}
C^{2n+1}(\mathfrak{L},V):={\rm Hom}(\underbrace{\mathfrak{L}\times...\times\mathfrak{L}}_{(2n+1)},V),~~(n\geq0),
\end{align*}
which is the space of $(2n+1)$-cochains. For any $f\in C^{2n+1}(\mathfrak{L},V)$ satisfies
\begin{align*}
f(x_1,x_2,...,x_{2n-2},x,x,y)=0,
\end{align*}
and
\begin{align*}
f(x_1,x_2,...,x_{2n-2},x,y,z)+f(x_1,x_2,...,x_{2n-2},y,z,x)+f(x_1,x_2,...,x_{2n-2},z,x,y)=0,
\end{align*}
for all $x,y,z,x_i\in \mathfrak{L},$ $i=1,2,...,2n-2.$

The coboundary operator $\delta:C^{2n-1}(\mathfrak{L},V)\rightarrow C^{2n+1}(\mathfrak{L},V)$ is given by
\begin{gather*}
\begin{aligned}
&(\delta f)(x_1,x_2,...,x_{2n+1})\\
=&~\theta(x_{2n},x_{2n+1})f(x_1,x_2,...,x_{2n-1})
-\theta(x_{2n-1},x_{2n+1})f(x_1,x_2,...,x_{2n})\\
&+\sum\limits_{i=1}^{n}(-1)^{i+1}D(x_{2i-1},x_{2i})f(x_1...,x_{2i-2},x_{2i+1},...,x_{2n+1})\\
&+\sum\limits_{i=1}^{n}\sum\limits_{j=2i+1}^{2n+1}
(-1)^{i+n+1}f(x_1...,x_{2i-2},x_{2i+1},...,[x_{2i-1},x_{2i},x_j],...,x_{2n+1}),\label{322.17}
\end{aligned}
\end{gather*}
for any $x_1,x_2,...,x_{2n+1}\in \mathfrak{L},$ $f\in C^{2n-1}(\mathfrak{L},V).$

It was proved in {\rm\cite{Yam}} that $\delta\circ\delta=0$. Thus, $(\underset{n=0}{\overset{+\infty}{\oplus}} C^{2n+1}(\mathfrak{L},V),\delta)$ is a cochain complex.

\begin{definition}{\rm\cite{Yam}}
The cohomology of the Lie triple system $\mathfrak{L}$ with coefficients in $V$ is the cohomology of the cochain complex $(\underset{n=0}{\overset{+\infty}{\oplus}} C^{2n+1}(\mathfrak{L},V),\delta)$. Denote by $Z^{2n+1}(\mathfrak{L},V)$ and $B^{2n+1}(\mathfrak{L},V)$ the set of $(2n+1)$-cocycles and the set of $(2n+1)$-coboundaries, respectively. The $(2n+1)$-th cohomology group is defined by
\begin{align*}
H^{2n+1}(\mathfrak{L},V)=Z^{2n+1}(\mathfrak{L},V)/B^{2n+1}(\mathfrak{L},V).
\end{align*}
\end{definition}

\subsection{Cohomologies of relative Rota-Baxter operators of weight $\lambda$ on Lie triple systems}

In this subsection, we establish a representation of the Lie triple system $(\mathfrak{L}',[\cdot,\cdot,\cdot]_T)$ on the vector space $\mathfrak{L}$ from a relative Rota-Baxter operator $T: \mathfrak{L}' \rightarrow \mathfrak{L}$ of weight $\lambda$ and define the cohomologies of relative Rota-Baxter operators of weight $\lambda$ on Lie triple systems.

\begin{lemma}\label{lem2113.2}
Let $T: \mathfrak{L}' \rightarrow \mathfrak{L}$ be a relative Rota-Baxter operator of weight $\lambda$ from a Lie triple system $(\mathfrak{L}',[\cdot,\cdot,\cdot]_{\mathfrak{L}'})$ to a Lie triple system $(\mathfrak{L},[\cdot,\cdot,\cdot]_{\mathfrak{L}})$ with respect to an action $\theta$. Define $\theta_T:\mathfrak{L}'\times\mathfrak{L}' \rightarrow {\rm End}(\mathfrak{L})$ by
\begin{align*}
\theta_T(u,v)x=[x,Tu,Tv]_{\mathfrak{L}}-T\Big(D(x,Tu)v-\theta(x,Tv)u\Big),
\end{align*}
also
\begin{align*}
D_T(u,v)x=\theta_T(v,u)x-\theta_T(u,v)x=[Tu,Tv,x]_{\mathfrak{L}}-T\Big(\theta(Tv,x)u-\theta(Tu,x)v\Big),
\end{align*}
for all $x\in \mathfrak{L}$, $u,v\in \mathfrak{L}'$. Then $(\mathfrak{L},\theta_T)$ is a representation of the descendent Lie triple system $(\mathfrak{L}',[\cdot,\cdot,\cdot]_T)$.
\end{lemma}
\begin{proof}
By calculation using the definitions of Lie triple system, representation and relative Rota-Baxter operator of weight $\lambda$, for all $u_i\in \mathfrak{L}'$, $i=1,2,3,4,$ $x\in \mathfrak{L}$, we have
\begin{align*}
&\Big(\theta_T(u_3,u_4)\theta_T(u_1,u_2)-\theta_T(u_2,u_4)\theta_T(u_1,u_3)-\theta_T(u_1,[u_2,u_3,u_4]_{\mathfrak{L}'})
+D_T(u_2,u_3)\theta_T(u_1,u_4)\Big)x\\
=&\theta_T(u_3,u_4)\Big([x,Tu_1,Tu_2]_{\mathfrak{L}}-T(D(x,Tu_1)u_2-\theta(x,Tu_2)u_1)\Big)\\
&-\theta_T(u_2,u_4)\Big([x,Tu_1,Tu_3]_{\mathfrak{L}}-T(D(x,Tu_1)u_3-\theta(x,Tu_3)u_1)\Big)\\
&-\Big([x,Tu_1,[Tu_2,Tu_3,Tu_4]_\mathfrak{L}]_\mathfrak{L}-T(D(x,Tu_1)[u_2,u_3,u_4]_{\mathfrak{L}'}
-\theta(x,[Tu_2,Tu_3,Tu_4]_\mathfrak{L})u_1\Big)\\
&+D_T(u_2,u_3)\Big([x,Tu_1,Tu_4]_{\mathfrak{L}}-T(D(x,Tu_1)u_4-\theta(x,Tu_4)u_1)\Big)\\
=&[[x,Tu_1,Tu_2]_{\mathfrak{L}},Tu_3,Tu_4]_{\mathfrak{L}}-T\Big(D([x,Tu_1,Tu_2]_{\mathfrak{L}},Tu_3)u_4
-\theta([x,Tu_1,Tu_2]_{\mathfrak{L}},Tu_4)u_3\Big)\\
&-[TD(x,Tu_1)u_2,Tu_3,Tu_4]_{\mathfrak{L}}+T\Big(D(TD(x,Tu_1)u_2,Tu_3)u_4-\theta(TD(x,Tu_1)u_2,Tu_4)u_3\Big)\\
&+[T\theta(x,Tu_2)u_1,Tu_3,Tu_4]_{\mathfrak{L}}-T\Big(D(T\theta(x,Tu_2)u_1,Tu_3)u_4-\theta(T\theta(x,Tu_2)u_1,Tu_4)u_3\Big)\\
&-[[x,Tu_1,Tu_3]_{\mathfrak{L}},Tu_2,Tu_4]_{\mathfrak{L}}+T\Big(D([x,Tu_1,Tu_3]_{\mathfrak{L}},Tu_2)u_4
-\theta([x,Tu_1,Tu_3]_{\mathfrak{L}},Tu_4)u_2\Big)\\
&+[TD(x,Tu_1)u_3,Tu_2,Tu_4]_{\mathfrak{L}}-T\Big(D(TD(x,Tu_1)u_3,Tu_2)u_4-\theta(TD(x,Tu_1)u_3,Tu_4)u_2\Big)\\
&-[T\theta(x,Tu_3)u_1,Tu_2,Tu_4]_{\mathfrak{L}}+T\Big(D(T\theta(x,Tu_3)u_1,Tu_2)u_4-\theta(T\theta(x,Tu_3)u_1,Tu_4)u_2\Big)\\
&+TD(x,Tu_1)\Big(D(Tu_2,Tu_3)u_4-\theta(Tu_2,Tu_4)u_3+\theta(Tu_3,Tu_4)u_2+\lambda[u_2,u_3,u_4]_{\mathfrak{L}'}\Big)\\
&-[x,Tu_1,[Tu_2,Tu_3,Tu_4]_{\mathfrak{L}}]_{\mathfrak{L}}-T\theta(x,[Tu_2,Tu_3,Tu_4]_{\mathfrak{L}})u_1\\
&+[Tu_2,Tu_3,[x,Tu_1,Tu_4]_{\mathfrak{L}}]_{\mathfrak{L}}-T\Big(\theta(Tu_3,[x,,Tu_1,Tu_4]_{\mathfrak{L}})u_2
-\theta(Tu_2,[x,,Tu_1,Tu_4]_{\mathfrak{L}})u_3\Big)\\
&-[Tu_2,Tu_3,TD(x,Tu_1)u_4]_{\mathfrak{L}}+T\Big(\theta(Tu_3,TD(x,Tu_1)u_4)u_2-\theta(Tu_2,TD(x,Tu_1)u_4)u_3\Big)\\
&+[Tu_2,Tu_3,T\theta(x,Tu_4)u_1]_{\mathfrak{L}}-T\Big(\theta(Tu_3,T\theta(x,Tu_4)u_1)u_2-\theta(Tu_2,T\theta(x,Tu_4)u_1)u_3\Big)\\
=&0,
\end{align*}
and
\begin{align*}
&\Big(\theta_T(u_3,u_4)D_T(u_1,u_2)-D_T(u_1,u_2)\theta_T(u_3,u_4)+\theta_T([u_1,u_2,u_3]_{\mathfrak{L}'},u_4)
+\theta_T(u_3,[u_1,u_2,u_4]_{\mathfrak{L}'})\Big)x\\
=&\theta_T(u_3,u_4)\Big([Tu_1,Tu_2,x]_\mathfrak{L}-T(\theta(Tu_2,x)u_1-\theta(Tu_1,x)u_2)\Big)\\
&-D_T(u_1,u_2)\Big([x,Tu_3,Tu_4]_\mathfrak{L}-T(D(x,Tu_3)u_4-\theta(x,Tu_4)u_3)\Big)\\
&+\Big([x,[Tu_1,Tu_2,Tu_3]_\mathfrak{L},Tu_4]_{\mathfrak{L}}-T(D(x,[Tu_1,Tu_2,Tu_3]_\mathfrak{L})u_4
-\theta(x,Tu_4)[u_1,u_2,u_3]_{\mathfrak{L}'})\Big)\\
&+\Big([x,Tu_3,[Tu_1,Tu_2,Tu_4]_\mathfrak{L}]_{\mathfrak{L}}-T(D(x,Tu_3)[u_1,u_2,u_4]_{\mathfrak{L}'}
-\theta(x,[Tu_1,Tu_2,Tu_4]_\mathfrak{L})u_3\Big)\\
=&[[Tu_1,Tu_2,x]_{\mathfrak{L}},Tu_3,Tu_4]_{\mathfrak{L}}-T\Big(D([Tu_1,Tu_2,x]_{\mathfrak{L}},Tu_3)u_4
-\theta([Tu_1,Tu_2,x]_{\mathfrak{L}},Tu_4)u_3\Big)\\
&-[T\theta(Tu_2,x)u_1,Tu_3,Tu_4]_{\mathfrak{L}}+T\Big(D(T\theta(Tu_2,x)u_1,Tu_3)u_4-\theta(T\theta(Tu_2,x)u_1,Tu_4)u_3\Big)\\
&+[T\theta(Tu_1,x)u_2,Tu_3,Tu_4]_{\mathfrak{L}}-T\Big(D(T\theta(Tu_1,x)u_2,Tu_3)u_4-\theta(T\theta(Tu_1,x)u_2,Tu_4)u_3\Big)\\
&-[Tu_1,Tu_2,[x,Tu_3,Tu_4]_{\mathfrak{L}}]_{\mathfrak{L}}+T\Big(\theta(Tu_2,[x,Tu_3,Tu_4]_{\mathfrak{L}})u_1
-\theta(Tu_1,[x,Tu_3,Tu_4]_{\mathfrak{L}})u_2\Big)\\
&+[Tu_1,Tu_2, TD(x,Tu_3)u_4]_{\mathfrak{L}}-T\Big(\theta(Tu_2,TD(x,Tu_3)u_4)u_1-\theta(Tu_1,TD(x,Tu_3)u_4)u_2\Big)\\
&-[Tu_1,Tu_2, T\theta(x,Tu_4)u_3]_{\mathfrak{L}}+T\Big(\theta(Tu_2,T\theta(x,Tu_4)u_3)u_1-\theta(Tu_1,T\theta(x,Tu_4)u_3)u_2\Big)\\
&+T\theta(x,Tu_4)\Big(D(Tu_1,Tu_2)u_3-\theta(Tu_1,Tu_3)u_2+\theta(Tu_2,Tu_3)u_1+\lambda[u_1,u_2,u_3]_{\mathfrak{L}'}\Big)\\
&+[x,[Tu_1,Tu_2,Tu_3]_{\mathfrak{L}},Tu_4]_{\mathfrak{L}}-TD(x,[Tu_1,Tu_2,Tu_3]_{\mathfrak{L}})u_4
+[x,Tu_3,[Tu_1,Tu_2,Tu_4]_{\mathfrak{L}}]_{\mathfrak{L}}\\
&-TD(x,Tu_3)\Big(D(Tu_1,Tu_2)u_4-\theta(Tu_1,Tu_4)u_2+\theta(Tu_1,Tu_4)u_2+\lambda[u_1,u_2,u_4]_{\mathfrak{L}'}\Big)\\
&+T\theta(x,[Tu_1,Tu_2,Tu_4]_{\mathfrak{L}})u_3\\
=&0.
\end{align*}
Thus, $(\mathfrak{L},\theta_T)$ is a representation of the descendent Lie triple system $(\mathfrak{L}',[\cdot,\cdot,\cdot]_T)$.
\end{proof}

Let $(\mathfrak{L},\theta_T)$ be a representation of a Lie triple system $(\mathfrak{L}',[\cdot,\cdot,\cdot]_{T})$. Denote by
\begin{align*}
C^{2n+1}_T(\mathfrak{L}',\mathfrak{L}):={\rm Hom}(\underbrace{\mathfrak{L}'\times...\times\mathfrak{L}'}_{(2n+1)},\mathfrak{L}),~~(n\geq0),
\end{align*}
which is the space of $(2n+1)$-cochains. For any $f\in C^{2n+1}_T(\mathfrak{L}',\mathfrak{L})$ satisfies
\begin{align*}
f(v_1,v_2,...,v_{2n-2},v,v,u)=0,
\end{align*}
and
\begin{align*}
f(v_1,v_2,...,v_{2n-2},v,u,w)+f(v_1,v_2,...,v_{2n-2},u,w,v)+f(v_1,v_2,...,v_{2n-2},w,v,u)=0,
\end{align*}
for all $u,v,w,v_i\in \mathfrak{L}'$, $i=1,2,...,2n-2$.

Let $\partial_T:C^{2n-1}_T(\mathfrak{L}',\mathfrak{L})\rightarrow C^{2n+1}_T(\mathfrak{L}',\mathfrak{L})$ be the corresponding coboundary operator of the Lie triple system $(\mathfrak{L}',[\cdot,\cdot,\cdot]_T)$ with coefficient in the representation $(\mathfrak{L},\theta_T)$. More precisely,
\begin{gather*}
\begin{aligned}
&(\partial_T f)(v_1,v_2,...,v_{2n+1})\\
=&~\theta_T(v_{2n},v_{2n+1})f(v_1,v_2,...,v_{2n-1})
-\theta_T(v_{2n-1},v_{2n+1})f(v_1,v_2,...,v_{2n})\\
&+\sum\limits_{i=1}^{n}(-1)^{i+1}D_T(v_{2i-1},v_{2i})f(v_1...,v_{2i-2},v_{2i+1},...,v_{2n+1})\\
&+\sum\limits_{i=1}^{n}\sum\limits_{j=2i+1}^{2n+1}
(-1)^{i+n+1}f(v_1...,v_{2i-2},v_{2i+1},...,[v_{2i-1},v_{2i},v_j]_T,...,v_{2n+1}),\label{322.17}
\end{aligned}
\end{gather*}
for any $v_1,v_2,...,v_{2n+1}\in \mathfrak{L}',$ $f\in C^{2n-1}_T(\mathfrak{L}',\mathfrak{L}).$

Note that, $f\in C^{1}_T(\mathfrak{L}',\mathfrak{L})$ is closed if and only if
\begin{gather}
\begin{aligned}\label{392.9}
&[f(v_1),Tv_2,Tv_3]_\mathfrak{L}+[Tv_1,f(v_2),Tv_3]_\mathfrak{L}+[Tv_1,Tv_2,f(v_3)]_\mathfrak{L}\\
&-T\Big(D(f(v_1),Tv_2)v_3-\theta(f(v_1),Tv_3)v_2+\theta(f(v_2),Tv_3)v_1\Big)\\
&-T\Big(\theta(Tv_2,f(v_3))v_1-\theta(Tv_1,f(v_3)v_2)-D(f(v_2),Tv_1)v_3\Big)\\
&-f\Big(\theta(Tv_2,Tv_3)v_1-\theta(Tv_1,Tv_3)v_2+D(Tv_1,Tv_2)v_3+[v_1,v_2,v_3]_{\mathfrak{L}'}\Big)=0.
\end{aligned}
\end{gather}

For any $\mathfrak{X}\in \mathfrak{L}\wedge\mathfrak{L},$ we define $\delta_T(\mathfrak{X}):\mathfrak{L}'\rightarrow \mathfrak{L}$ by
\begin{align}\label{392.10}
\delta_T(\mathfrak{X})v=TD(\mathfrak{X})v-[\mathfrak{X},Tv]_\mathfrak{L},~~\forall ~v\in \mathfrak{L}'.
\end{align}

\begin{proposition}\label{prop2113.3}
Let $T$ be a relative Rota-Baxter operator of weight $\lambda$ from $(\mathfrak{L},[\cdot,\cdot,\cdot]_\mathfrak{L})$ to
$(\mathfrak{L}',[\cdot,\cdot,\cdot]_{\mathfrak{L}'})$ with respect to an action $\theta$. Then $\delta_T(\mathfrak{X})$ is a $1$-cocycle of the Lie triple system $(\mathfrak{L}',[\cdot,\cdot,\cdot]_T)$ with coefficient $(\mathfrak{L},\theta_T).$
\end{proposition}
\begin{proof}
For any $u_1,u_2,u_3\in V,$ we have
\small{\begin{align*}
&\partial_T\delta_T(\mathfrak{X})(u_1,u_2,u_3)\\
\overset{(\ref{392.9})}{=}&[Tu_1,Tu_2,\delta_T(\mathfrak{X})u_3]_\mathfrak{L}+[\delta_T(\mathfrak{X})u_1,Tu_2,Tu_3]_\mathfrak{L}
+[Tu_1,\delta_T(\mathfrak{X})u_2,Tu_3]_\mathfrak{L}\\
&-T\Big(D(\delta_T(\mathfrak{X})u_1,Tu_2)u_3-\theta(\delta_T(\mathfrak{X})u_1,Tu_3)u_2\Big)
+T\Big(D(\delta_T(\mathfrak{X})u_2,Tu_1)u_3-\theta(\delta_T(\mathfrak{X})u_2,Tu_3)u_1\Big)\\
&-T\Big(\theta(Tu_2,\delta_T(\mathfrak{X})u_3)u_1-\theta(Tu_1,\delta_T(\mathfrak{X})u_3)u_2\Big)
-\delta_T(\mathfrak{X})\Big(\theta(Tu_2,Tu_3)u_1-\theta(Tu_1,Tu_3)u_2\\
&+D(Tu_1,Tu_2)u_3+\lambda[u_1,u_2,u_3]_{\mathfrak{L}'}\Big)\\
\overset{(\ref{392.10})}{=}&[Tu_1,Tu_2,TD(\mathfrak{X})u_3]_\mathfrak{L}-[Tu_1,Tu_2,[\mathfrak{X},Tu_3]_\mathfrak{L}]_\mathfrak{L}
+[TD(\mathfrak{X})u_1,Tu_2,Tu_3]_\mathfrak{L}-[[\mathfrak{X},Tu_1]_\mathfrak{L},Tu_2,Tu_3]_\mathfrak{L}\\
&+[Tu_1,TD(\mathfrak{X})u_2,Tu_3]_\mathfrak{L}-[Tu_1,[\mathfrak{X},Tu_2]_\mathfrak{L},Tu_3]_\mathfrak{L}
-T\Big(D(TD(\mathfrak{X})u_1,Tu_2)u_3-D([\mathfrak{X},Tu_1]_\mathfrak{L},Tu_2)u_3\Big)\\
&+T\Big(\theta(TD(\mathfrak{X})u_1,Tu_3)u_2-\theta([\mathfrak{X},Tu_1]_\mathfrak{L},Tu_3)u_2\Big)
+T\Big(D(TD(\mathfrak{X})u_2,Tu_1)u_3-D([\mathfrak{X},Tu_2]_\mathfrak{L},Tu_1)u_3\Big)\\
&-T\Big(\theta(TD(\mathfrak{X})u_2,Tu_3)u_1-\theta([\mathfrak{X},Tu_2]_\mathfrak{L},Tu_3)u_1\Big)
-T\Big(\theta(Tu_2,TD(\mathfrak{X})u_3)u_1-\theta(Tu_2,[\mathfrak{X},Tu_3]_\mathfrak{L})u_1\Big)\\
&+T\Big(\theta(Tu_1,TD(\mathfrak{X})u_3)u_2-\theta(Tu_1,[\mathfrak{X},Tu_3]_\mathfrak{L})u_2\Big)
-TD(\mathfrak{X})\theta(Tu_2,Tu_3)u_1+TD(\mathfrak{X})\theta(Tu_1,Tu_3)u_2\\
&-TD(\mathfrak{X})D(Tu_1,Tu_2)u_3+[\mathfrak{X},T\theta(Tu_2,Tu_3)u_1]_\mathfrak{L}-[\mathfrak{X},T\theta(Tu_1,Tu_3)u_2]_\mathfrak{L}
+[\mathfrak{X},TD(Tu_1,Tu_2)u_3]_\mathfrak{L}\\
&+\lambda TD(\mathfrak{X})[u_1,u_2,u_3]_{\mathfrak{L}'}-\lambda [\mathfrak{X},[u_1,u_2,u_3]_{\mathfrak{L}'}]_\mathfrak{L}\\
\overset{(\ref{2112.7})}{=}&-TD(\mathfrak{X})\theta(Tu_2,Tu_3)u_1+TD(\mathfrak{X})\theta(Tu_1,Tu_3)u_2
-TD(\mathfrak{X})\theta(Tu_1,Tu_2)u_3\\
&+TD([\mathfrak{X},Tu_1]_\mathfrak{L},Tu_2)u_3-T\theta([\mathfrak{X},Tu_1]_\mathfrak{L},Tu_3)u_2
-TD([\mathfrak{X},Tu_2]_\mathfrak{L},Tu_1)u_3\\
&+T\theta([\mathfrak{X},Tu_2]_\mathfrak{L},Tu_3)u_1+T\theta(Tu_2,[\mathfrak{X},Tu_3]_\mathfrak{L})u_1
-T\theta(Tu_1,[\mathfrak{X},Tu_3]_\mathfrak{L})u_2\\
&+TD(Tu_1,Tu_2)D(\mathfrak{X})u_3+T\theta(Tu_2,Tu_3)D(\mathfrak{X})u_1-T\theta(Tu_1,Tu_3)D(\mathfrak{X})u_2\\
\overset{Def.\ref{def392.2}}{=}&0,
\end{align*}}
which implies that $\delta_T(\mathfrak{X})$ is a $1$-cocycle.
\end{proof}

Next, we introduce the cohomology theory of relative Rota-Baxter operators of weight $\lambda$ on Lie triple systems.

Let $T$ be a relative Rota-Baxter operator of weight $\lambda$ from a Lie triple system $(\mathfrak{L}',[\cdot,\cdot,\cdot]_{\mathfrak{L}'})$ to a Lie triple system $(\mathfrak{L},[\cdot,\cdot,\cdot]_\mathfrak{L})$ with respect to an action $\theta$. Define the space of $(2n-1)$-cochains by
\begin{equation}\label{2113.2}
C^{2n-1}_T(\mathfrak{L}',\mathfrak{L})=\left\{
\begin{array}{lr}C^{2n-1}(\mathfrak{L}',\mathfrak{L}),~~&n\geq1,\\
\mathfrak{L}\wedge\mathfrak{L},~~&n=0.
\end{array}
\right.
\end{equation}

Define $\partial:C^{2n-1}_T(\mathfrak{L}',\mathfrak{L})\rightarrow C^{2n+1}_T(\mathfrak{L}',\mathfrak{L})$ by
\begin{equation}\label{2113.3}
\partial=\left\{
\begin{array}{lr}\partial_T,~~&n\geq1,\\
\delta_T,~~&n=0.
\end{array}
\right.
\end{equation}

\begin{theorem}
$(\underset{n=0}{\overset{+\infty}{\oplus}} C^{2n-1}_T(\mathfrak{L}',\mathfrak{L}),\partial)$ is a cochain complex.
\end{theorem}
\begin{proof}
It follows from Proposition \ref{prop2113.3} and the fact that $\delta_T$ is the corresponding coboundary operator of the descendent Lie triple system $(\mathfrak{L}',[\cdot,\cdot,\cdot]_T)$ with coefficient in the representation $(\mathfrak{L},\theta_T)$ directly.
\end{proof}

\begin{definition}
The coboundary of the cochain complex $(\underset{n=0}{\overset{+\infty}{\oplus}} C^{2n-1}_T(\mathfrak{L}',\mathfrak{L}),\partial)$ is taken to be the cohomology for the relative Rota-Baxter operator $T$ of weight $\lambda$. Denote the set of $(2n-1)$-cocycles by $Z_T^{2n-1}(\mathfrak{L}',\mathfrak{L})$, the set of $(2n-1)$-coboundaries by $B_T^{2n-1}(\mathfrak{L}',\mathfrak{L})$ and $(2n-1)$-th cohomology group by
\begin{align}\label{2113.70}
H_T^{2n-1}(\mathfrak{L}',\mathfrak{L})=Z_T^{2n-1}(\mathfrak{L}',\mathfrak{L})/B_T^{2n-1}(\mathfrak{L}',\mathfrak{L}).
\end{align}
\end{definition}

\subsection{Infinitesimal deformations of relative Rota-Baxter operators}
In this subsection, we use the cohomology theory to characterize infinitesimal deformations of relative Rota-Baxter operators of weight $\lambda$ on Lie triple systems.

Let $(\mathfrak{L},[\cdot,\cdot,\cdot]_\mathfrak{L})$ be a Lie triple system over $\mathbb{F}$ and $\mathbb{F}[t]$ be the polynomial ring in one variable $t$. Then $\mathbb{F}[t]/(t^2)\otimes_\mathbb{F}\mathfrak{L}$ is an $\mathbb{F}[t]/(t^2)$-module. Here we denote $f(t)\otimes_\mathbb{F}x\in\mathbb{F}[t]/(t^2)\otimes_\mathbb{F}\mathfrak{L}$ by $f(t)x$, where $f(t)\in \mathbb{F}[t]/(t^2)$ and all the vector spaces are finite dimensional vector spaces over $\mathbb{F}$. Moreover, $\mathbb{F}[t]/(t^2)\otimes_\mathbb{F}\mathfrak{L}$ is a Lie triple system over $\mathbb{F}[t]/(t^2)$, where the Lie triple system is defined by
\begin{align*}
[f_1(t)x_1,f_2(t)x_2,f_3(t)x_3]=f_1(t)f_2(t)f_3(t)[x_1,x_2,x_3]_\mathfrak{L},
\end{align*}
for $f_i(t)\in \mathbb{F}[t]/(t^2)$, $1\leq i \leq3$, $x_1,x_2,x_3\in \mathfrak{L}$.

\begin{definition}
Let $T:\mathfrak{L}'\rightarrow \mathfrak{L}$ be a relative Rota-Baxter operator of weight $\lambda$ from a Lie triple system $(\mathfrak{L}',[\cdot,\cdot,\cdot]_{\mathfrak{L}'})$ to a Lie triple system $(\mathfrak{L},[\cdot,\cdot,\cdot]_\mathfrak{L})$ with respect to an action $\theta$. Let $\mathfrak{T}:\mathfrak{L}' \rightarrow \mathfrak{L}$ be a linear map. If $T_t=T+t\mathfrak{T}$ is still a relative Rota-Baxter operator of weight $\lambda$ for all $t$, we say that $\mathfrak{T}$ generates an infinitesimal deformation of $T$.
\end{definition}

Since $T_t=T+t\mathfrak{T}$ is a relative Rota-Baxter operator of weight $\lambda$, by consider the coefficients of $t^n$, for any $u,v,w\in \mathfrak{L}'$. We have
\begin{gather}\label{2113.4}
\begin{aligned}
&[\mathfrak{T}u,Tv,Tw]_\mathfrak{L}+[Tu,\mathfrak{T}v,Tw]_\mathfrak{L}+[Tu,Tv,\mathfrak{T}w]_\mathfrak{L}\\
=&T\Big(\theta(Tv,\mathfrak{T}w)u-\theta(Tu,\mathfrak{T}w)v+D(\mathfrak{T}u, Tv)w\\
&+\theta(\mathfrak{T}v, Tw)u-\theta(\mathfrak{T}u, Tw)v+D(Tu,\mathfrak{T}v)w\Big)\\
&+\mathfrak{T}\Big(D(Tu,Tv)w+\theta(Tv,Tw)u-\theta(Tu,Tw)v+\lambda[u,v,w]_{\mathfrak{L}'}\Big),
\end{aligned}
\end{gather}
\begin{gather*}\label{2113.5}
\begin{aligned}
&[\mathfrak{T}u,\mathfrak{T}v,Tw]_\mathfrak{L}+[Tu,\mathfrak{T}v,\mathfrak{T}w]_\mathfrak{L}
+[\mathfrak{T}u,Tv,\mathfrak{T}w]_\mathfrak{L}\\
=&\mathfrak{T}\Big(\theta(\mathfrak{T}v,Tw)u-\theta(\mathfrak{T}u,Tw)v+D(Tu, \mathfrak{T}v)w+\theta(Tv, \mathfrak{T}w)u-\theta(Tu, \mathfrak{T}w)v+D(\mathfrak{T}u,Tv)w\Big)\\
&+T\Big(D(\mathfrak{T}u,\mathfrak{T}v)w+\theta(\mathfrak{T}v,\mathfrak{T}w)u-\theta(\mathfrak{T}u,\mathfrak{T}w)v\Big),
\end{aligned}
\end{gather*}
\begin{align*}
[\mathfrak{T}u,\mathfrak{T}v,\mathfrak{T}w]_\mathfrak{L}=\mathfrak{T}\Big(D(\mathfrak{T}u,\mathfrak{T}v)w
-\theta(\mathfrak{T}u,\mathfrak{T}w)v+\theta(\mathfrak{T}v,\mathfrak{T}w)u\Big).
\end{align*}

Note that Eq. (\ref{2113.4}) means that $\mathfrak{T}$ is a $1$-cocycle of the relative Rota-Baxter operator $T$. Hence $\mathfrak{T}$ defines a cohomology class in $H^1_T(\mathfrak{L}',\mathfrak{L})$. 

\begin{definition}
Let $T$ be a relative Rota-Baxter operator of weight $\lambda$ from a Lie triple system $(\mathfrak{L}',[\cdot,\cdot,\cdot]_{\mathfrak{L}'})$ to a Lie triple system $(\mathfrak{L},[\cdot,\cdot,\cdot]_\mathfrak{L})$ with respect to an action $\theta$. Two 1-parameter infinitesimal deformations $T_t^1=T+t\mathfrak{T}_1$ and $T_t^2=T+t\mathfrak{T}_2$ are said to be equivalent if there exist $\mathfrak{X}\in \mathfrak{L}\wedge\mathfrak{L}$ such that $({\rm id}_{\mathfrak{L}}+t[\mathfrak{X},-],{\rm id}_{\mathfrak{L}'}+tD(\mathfrak{X}))$ is a homomorphism from $T_t^1$ to $T_t^2$. In particular, an infinitesimal deformation $T_t=T+t\mathfrak{T}$ of a relative Rota-Baxter operator $T$ is said to be trivial if there exists $\mathfrak{X}\in \mathfrak{L}\wedge\mathfrak{L}$ such that $({\rm id}_{\mathfrak{L}}+t[\mathfrak{X},-],{\rm id}_{\mathfrak{L}'}+tD(\mathfrak{X}))$ is a homomorphism from $T_t$ to $T$.
\end{definition}

Let $({\rm id}_{\mathfrak{L}}+t[\mathfrak{X},-],{\rm id}_{\mathfrak{L}'}+tD(\mathfrak{X}))$ is a homomorphism from $T_t^1$ to $T_t^2$. By Eq. (\ref{2112.8}) we get,
\begin{align*}
({\rm id}_{\mathfrak{L}}+t[\mathfrak{X},-])(T+t\mathfrak{T}_1)(u)=(T+t\mathfrak{T}_2)({\rm id}_{\mathfrak{L}'}+tD(\mathfrak{X}))(u),
\end{align*}
which implies
\begin{align}\label{2113.7}
\left\{
\begin{array}{lr}\mathfrak{T}_1(u)-\mathfrak{T}_2(u)=TD(\mathfrak{X})(u)-[\mathfrak{X},Tu]_\mathfrak{L},~~&\\
~~[\mathfrak{X},\mathfrak{T}_1(u)]_\mathfrak{L}=\mathfrak{T}_2D(\mathfrak{X})(u).&
\end{array}
\right.
\end{align}
 Now we are ready to give the main result in this section.

\begin{theorem}
Let $T$ be a relative Rota-Baxter operator of weight $\lambda$ from a Lie triple system $(\mathfrak{L}',[\cdot,\cdot,\cdot]_{\mathfrak{L}'})$ to a Lie triple system $(\mathfrak{L},[\cdot,\cdot,\cdot]_\mathfrak{L})$ with respect to an action $\theta$. If two 1-parameter infinitesimal deformations $T_t^1=T+t\mathfrak{T}_1$ and $T_t^2=T+t\mathfrak{T}_2$ are equivalent, then $\mathfrak{T}_1$ and $\mathfrak{T}_2$ are in the same cohomology class in $H^1_T(\mathfrak{L}',\mathfrak{L})$.
\end{theorem}
\begin{proof}
It is easy to see from the first condition of Eq. (\ref{2113.7}) that
\begin{align*}
\mathfrak{T}_1(u)=\mathfrak{T}_2(u)+(\partial\mathfrak{X})(u),~~~\forall~u\in \mathfrak{L}',
\end{align*}
which implies that $\mathfrak{T}_1$ and $\mathfrak{T}_2$ are in the same cohomology class.
\end{proof}

\subsection{Functorial properties}

In this subsection, we show that the cohomology theory for relative Rota-Baxter operators of weight $\lambda$ on Lie triple systems enjoys certain functorial properties.

\begin{proposition}\label{prop2113.9}
Let $T$ and $T'$ be relative Rota-Baxter operators of weight $\lambda$ from a Lie triple system $(\mathfrak{L}',[\cdot,\cdot,\cdot]_{\mathfrak{L}'})$ to a Lie triple system $(\mathfrak{L},[\cdot,\cdot,\cdot]_\mathfrak{L})$ with respect to an action $\theta$. Let $(\psi_\mathfrak{L},\psi_{\mathfrak{L}'})$ be a homomorphism from $T$ to $T'$.

{\rm (i)} $\psi_{\mathfrak{L}'}$ is also a homomorphism from the descendent Lie triple system $(\mathfrak{L}',[\cdot,\cdot,\cdot]_T)$ of $T$ to the descendent Lie triple system $(\mathfrak{L}',[\cdot,\cdot,\cdot]_{T'})$ of $T'$.

{\rm (ii)} The induced representation $(\mathfrak{L},\theta_T)$ of the Lie triple system $(\mathfrak{L}',[\cdot,\cdot,\cdot]_T)$ and the induced representation $(\mathfrak{L},\theta_{T'})$ of the Lie triple system $(\mathfrak{L}',[\cdot,\cdot,\cdot]_{T'})$ satisfies the following relations$:$
\begin{align*}
\psi_\mathfrak{L}\theta_T(u,v)=\theta_{T'}(\psi_{\mathfrak{L}'}(u),\psi_{\mathfrak{L}'}(v))\psi_{\mathfrak{L}},
\end{align*}
and
\begin{align*}
\psi_\mathfrak{L} D_T(u,v)=D_{T'}(\psi_{\mathfrak{L}'}(u),\psi_{\mathfrak{L}'}(v))\psi_{\mathfrak{L}},
\end{align*}
for all $u,v\in \mathfrak{L}'$.

That is, the following diagram commutes$:$
\begin{equation*}
\begin{CD}
\mathfrak{L}@>\psi_\mathfrak{L}>> \mathfrak{L} \\
@V\theta_T(u,v)VV@VV\theta_{T'}(\psi_{\mathfrak{L}'}(u),\psi_{\mathfrak{L}'}(v)) V\\
\mathfrak{L}\ @>>\psi_\mathfrak{L}>\mathfrak{L}.
\end{CD}
\end{equation*}
\end{proposition}
\begin{proof}
{\rm (i)} By Eqs. (\ref{2112.8})-(\ref{2112.10}) and Lemma \ref{lem2113.2}, we have
\begin{align*}
\psi_{\mathfrak{L}'}([u,v,w]_T)=&\psi_{\mathfrak{L}'}\Big(D(Tu,Tv)w-\theta(Tu,Tw)v+\theta(Tv,Tw)u+\lambda[u,v,w]_{\mathfrak{L}'}\Big)\\
=&D\Big(\psi_\mathfrak{L}Tu,\psi_\mathfrak{L}Tv\Big)\psi_{\mathfrak{L}'}(w)
-\theta\Big(\psi_\mathfrak{L}Tu,\psi_\mathfrak{L}Tw\Big)\psi_{\mathfrak{L}'}(v)\\
&+\theta\Big(\psi_\mathfrak{L}Tv,\psi_\mathfrak{L}Tw\Big)\psi_{\mathfrak{L}'}(u)
+\lambda\psi_{\mathfrak{L}'}([u,v,w]_{\mathfrak{L}'})\\
=&D\Big(T'\psi_{\mathfrak{L}'}(u),T'\psi_{\mathfrak{L}'}(v)\Big)\psi_{\mathfrak{L}'}(w)
-\theta\Big(T'\psi_{\mathfrak{L}'}(u),T'\psi_{\mathfrak{L}'}(w)\Big)\psi_{\mathfrak{L}'}(v)\\
&+\theta\Big(T'\psi_{\mathfrak{L}'}(v),T'\psi_{\mathfrak{L}'}(w)\Big)\psi_{\mathfrak{L}'}(u)
+\lambda[\psi_{\mathfrak{L}'}(u),\psi_{\mathfrak{L}'}(v),\psi_{\mathfrak{L}'}(w)]_{\mathfrak{L}'}\\
=&[\psi_{\mathfrak{L}'}(u),\psi_{\mathfrak{L}'}(v),\psi_{\mathfrak{L}'}(w)]_{T'}.
\end{align*}

{\rm (ii)} Now, according to Eqs. (\ref{2112.8})-(\ref{2112.10}) and Lemma \ref{lem2113.2}, for all $u,v\in \mathfrak{L}'$, $x\in \mathfrak{L}$, we have
\begin{align*}
\psi_\mathfrak{L}\Big(\theta_T(u,v)x\Big)=&\psi_\mathfrak{L}\Big([x,Tu,Tv]_\mathfrak{L}-T(D(x,Tu)v-\theta(x,Tv)u)\Big)\\
=&[\psi_\mathfrak{L}(x),\psi_\mathfrak{L}(Tu),\psi_\mathfrak{L}(Tv)]_\mathfrak{L}
-T'\Big(\psi_{\mathfrak{L}'}(D(x,Tu)v)-\psi_{\mathfrak{L}'}(\theta(x,Tv)u)\Big)\\
=&[\psi_\mathfrak{L}(x),T'\psi_{\mathfrak{L}'}(u),T'\psi_{\mathfrak{L}'}(v)]_\mathfrak{L}
-T'\Big(D(\psi_\mathfrak{L}(x),T'\psi_{\mathfrak{L}'}(u))\psi_{\mathfrak{L}'}(v))\\
&-\theta(\psi_\mathfrak{L}(x),T'\psi_{\mathfrak{L}'}(v))\psi_{\mathfrak{L}'}(u))\Big)\\
=&\theta_T'(\psi_{\mathfrak{L}'}(u),\psi_{\mathfrak{L}'}(v))\psi_\mathfrak{L}(x).
\end{align*}

Similarly, we have 
\begin{align*}
\psi_\mathfrak{L} D_T(u,v)=D_{T'}(\psi_{\mathfrak{L}'}(u),\psi_{\mathfrak{L}'}(v))\psi_{\mathfrak{L}},
\end{align*}
for all $u,v\in \mathfrak{L}'$. We finish the proof.
\end{proof}

Let $T$ and $T'$ be relative Rota-Baxter operators of weight $\lambda$ from a Lie triple system $(\mathfrak{L}',[\cdot,\cdot,\cdot]_{\mathfrak{L}'})$ to a Lie triple system $(\mathfrak{L},[\cdot,\cdot,\cdot]_\mathfrak{L})$ with respect to an action $\theta$. Let $(\psi_\mathfrak{L},\psi_{\mathfrak{L}'})$ be a homomorphism from $T$ to $T'$ in which $\psi_{\mathfrak{L}'}$ is invertible. Define a map
\begin{align*}
p:C^{2n-1}_T(\mathfrak{L}',\mathfrak{L})\rightarrow C^{2n-1}_{T'}(\mathfrak{L}',\mathfrak{L})
\end{align*}
by
\begin{align*}
p(\omega)(u_1,u_2,..., u_{2n-1})
=\psi_\mathfrak{L}\Big(\omega(\psi^{-1}_{\mathfrak{L}'}(u_1),\psi^{-1}_{\mathfrak{L}'}(u_2),..., \psi^{-1}_{\mathfrak{L}'}(u_{2n-1}))\Big),
\end{align*}
for all $\omega\in C^{2n-1}_T(\mathfrak{L}',\mathfrak{L})$, $u_i\in \mathfrak{L}'$, $1=1,2,...,2n-1$.

\begin{theorem}
Keep the same notations as above, $p$ is a cochain map from the cochain complex $(\underset{n=1}{\overset{+\infty}{\oplus}} C^{2n-1}_T(\mathfrak{L}',\mathfrak{L}),\partial_T)$ to the cochain complex $(\underset{n=1}{\overset{+\infty}{\oplus}} C^{2n-1}_{T'}(\mathfrak{L}',\mathfrak{L}),\partial_{T'})$. Consequently, it induces a homomorphism $p_*$ from the cohomology group $H^{2n-1}_T(\mathfrak{L}',\mathfrak{L})$ to $H^{2n-1}_{T'}(\mathfrak{L}',\mathfrak{L})$.
\end{theorem}
\begin{proof}
For all $\omega\in C^{2n-1}_T(\mathfrak{L}',\mathfrak{L})$, by Proposition \ref{prop2113.9}, and Eqs. (\ref{2113.2})-(\ref{2113.70}), we have
\begin{align*}
&(\partial_{T'}p(\omega))(u_1,u_2,..., u_{2n-1},u_{2n},u_{2n+1})\\
=&\theta_{T'}(u_{2n},u_{2n+1})p(\omega)(u_1,u_2,..., u_{2n-1})-\theta_{T'}(u_{2n-1},u_{2n+1})p(\omega)(u_1,u_2,..., u_{2n-2},u_{2n})\\
&+\sum\limits_{k=1}^{n}(-1)^{n+k}D_{T'}(u_{2k-1},u_{2k})p(\omega)(u_1,u_2,...,\widehat{u_{2k-1}},\widehat{u_{2k}},..., u_{2n+1})\\
&+\sum\limits_{k=1}^{n}\sum\limits_{j=2k+1}^{2n+1}(-1)^{n+k+1}p(\omega)(u_1,u_2,...,\widehat{u_{2k-1}},\widehat{u_{2k}},...,
[u_{2k-1},u_{2k},u_j]_{T'},..., u_{2n+1})\\
=&\theta_{T'}(u_{2n},u_{2n+1})\psi_\mathfrak{L}\omega(\psi^{-1}_{\mathfrak{L}'}(u_1),\psi^{-1}_{\mathfrak{L}'}(u_2),..., \psi^{-1}_{\mathfrak{L}'}(u_{2n-1}))\\
&-\theta_{T'}(u_{2n-1},u_{2n+1})\psi_\mathfrak{L}\omega(\psi^{-1}_{\mathfrak{L}'}(u_1),
\psi^{-1}_{\mathfrak{L}'}(u_2),..., \psi^{-1}_{\mathfrak{L}'}(u_{2n-2}),\psi^{-1}_{\mathfrak{L}'}(u_{2n}))\\
&+\sum\limits_{k=1}^{n}(-1)^{n+k}D_{T'}(u_{2k-1},u_{2k})\psi_\mathfrak{L}\omega
(\psi^{-1}_{\mathfrak{L}'}(u_1),\psi^{-1}_{\mathfrak{L}'}(u_2),...,\widehat{\psi^{-1}_{\mathfrak{L}'}(u_{2k-1})},
\widehat{\psi^{-1}_{\mathfrak{L}'}(u_{2k})},..., \psi^{-1}_{\mathfrak{L}'}(u_{2n+1}))\\
&+\sum\limits_{k=1}^{n}\sum\limits_{j=2k+1}^{2n+1}(-1)^{n+k+1}\psi_\mathfrak{L}\omega
(\psi^{-1}_{\mathfrak{L}'}(u_1),\psi^{-1}_{\mathfrak{L}'}(u_2),...,\widehat{\psi^{-1}_{\mathfrak{L}'}(u_{2k-1})},
\widehat{\psi^{-1}_{\mathfrak{L}'}(u_{2k})},...,\\
&[\psi^{-1}_{\mathfrak{L}'}(u_{2k-1}),\psi^{-1}_{\mathfrak{L}'}(u_{2k}),\psi^{-1}_{\mathfrak{L}'}(u_j)]_{T},..., \psi^{-1}_{\mathfrak{L}'}(u_{2n+1}))\\
=&\theta_{T'}(\psi_{\mathfrak{L}'}\psi^{-1}_{\mathfrak{L}'}(u_{2n}),
\psi_{\mathfrak{L}'}\psi^{-1}_{\mathfrak{L}'}(u_{2n+1}))
\psi_\mathfrak{L}\omega(\psi^{-1}_{\mathfrak{L}'}(u_1),\psi^{-1}_{\mathfrak{L}'}(u_2),..., \psi^{-1}_{\mathfrak{L}'}(u_{2n-1}))\\
&-\theta_{T'}(\psi_{\mathfrak{L}'}\psi^{-1}_{\mathfrak{L}'}(u_{2n-1}),
\psi_{\mathfrak{L}'}\psi^{-1}_{\mathfrak{L}'}(u_{2n+1}))\psi_\mathfrak{L}\omega(\psi^{-1}_{\mathfrak{L}'}(u_1),
\psi^{-1}_{\mathfrak{L}'}(u_2),..., \psi^{-1}_{\mathfrak{L}'}(u_{2n-2}),\psi^{-1}_{\mathfrak{L}'}(u_{2n}))\\
&+\sum\limits_{k=1}^{n}(-1)^{n+k}D_{T'}(\psi_{\mathfrak{L}'}\psi^{-1}_{\mathfrak{L}'}(u_{2k-1}),
\psi_{\mathfrak{L}'}\psi^{-1}_{\mathfrak{L}'}(u_{2k}))\psi_\mathfrak{L}\omega
(\psi^{-1}_{\mathfrak{L}'}(u_1),\psi^{-1}_{\mathfrak{L}'}(u_2),...,\\
&\widehat{\psi^{-1}_{\mathfrak{L}'}(u_{2k-1})},
\widehat{\psi^{-1}_{\mathfrak{L}'}(u_{2k})},..., \psi^{-1}_{\mathfrak{L}'}(u_{2n+1}))\\
&+\sum\limits_{k=1}^{n}\sum\limits_{j=2k+1}^{2n+1}(-1)^{n+k+1}\psi_\mathfrak{L}\omega
(\psi^{-1}_{\mathfrak{L}'}(u_1),\psi^{-1}_{\mathfrak{L}'}(u_2),...,\widehat{\psi^{-1}_{\mathfrak{L}'}(u_{2k-1})},
\widehat{\psi^{-1}_{\mathfrak{L}'}(u_{2k})},...,\\
&[\psi^{-1}_{\mathfrak{L}'}(u_{2k-1}),\psi^{-1}_{\mathfrak{L}'}(u_{2k}),\psi^{-1}_{\mathfrak{L}'}(u_j)]_{T},..., \psi^{-1}_{\mathfrak{L}'}(u_{2n+1}))\\
=&\psi_{\mathfrak{L}}\theta_{T}(\psi^{-1}_{\mathfrak{L}'}(u_{2n}),\psi^{-1}_{\mathfrak{L}'}(u_{2n+1}))
\omega(\psi^{-1}_{\mathfrak{L}'}(u_1),\psi^{-1}_{\mathfrak{L}'}(u_2),..., \psi^{-1}_{\mathfrak{L}'}(u_{2n-1}))\\
&-\psi_{\mathfrak{L}}\theta_{T}(\psi^{-1}_{\mathfrak{L}'}(u_{2n-1}),\psi^{-1}_{\mathfrak{L}'}(u_{2n+1}))
\omega(\psi^{-1}_{\mathfrak{L}'}(u_1),\psi^{-1}_{\mathfrak{L}'}(u_2),..., \psi^{-1}_{\mathfrak{L}'}(u_{2n-2}),\psi^{-1}_{\mathfrak{L}'}(u_{2n}))\\
&+\sum\limits_{k=1}^{n}(-1)^{n+k}\psi_{\mathfrak{L}}D_{T}(\psi^{-1}_{\mathfrak{L}'}(u_{2k-1}),\psi^{-1}_{\mathfrak{L}'}(u_{2k}))
\omega(\psi^{-1}_{\mathfrak{L}'}(u_1),\psi^{-1}_{\mathfrak{L}'}(u_2),...,\\
&\widehat{\psi^{-1}_{\mathfrak{L}'}(u_{2k-1})},
\widehat{\psi^{-1}_{\mathfrak{L}'}(u_{2k})},..., \psi^{-1}_{\mathfrak{L}'}(u_{2n+1}))\\
&+\sum\limits_{k=1}^{n}\sum\limits_{j=2k+1}^{2n+1}(-1)^{n+k+1}\omega
(\psi^{-1}_{\mathfrak{L}'}(u_1),\psi^{-1}_{\mathfrak{L}'}(u_2),...,\widehat{\psi^{-1}_{\mathfrak{L}'}(u_{2k-1})},
\widehat{\psi^{-1}_{\mathfrak{L}'}(u_{2k})},...,\\
&[\psi^{-1}_{\mathfrak{L}'}(u_{2k-1}),\psi^{-1}_{\mathfrak{L}'}(u_{2k}),\psi^{-1}_{\mathfrak{L}'}(u_j)]_{T},..., \psi^{-1}_{\mathfrak{L}'}(u_{2n+1}))\\
=&\psi_{\mathfrak{L}}(\partial_T\omega(\psi^{-1}_{\mathfrak{L}'}(u_1),...,\psi^{-1}_{\mathfrak{L}'}(u_{2n+1})))\\
=&p(\partial_T\omega)(u_1,...,u_{2n+1}),
\end{align*}
where $u_i\in \mathfrak{L}'$, $i=1,2,...,2n+1$. Thus $p$ is a cochain map, and induces  a homomorphism $p_*$ from the cohomology group $H^{2n-1}_T(\mathfrak{L}',\mathfrak{L})$ to $H^{2n-1}_{T'}(\mathfrak{L}',\mathfrak{L})$.
\end{proof}

\end{document}